\documentclass[11pt]{amsart}
\usepackage{amsmath}
\usepackage{amsthm}
\usepackage{amssymb}
\usepackage{amscd}
\usepackage{amsbsy}
\usepackage{epsfig}

\theoremstyle{plain}
\newtheorem{thm}{Theorem}%[section]

\newtheorem{lem}{Lemma}%[section]
\newtheorem{prop}{Proposition}%[section]

\newtheorem{theoalph}{Theorem}

\theoremstyle{definition}
\newtheorem{definition}{Definition}

\renewcommand{\=}{ : = }

\newcommand{\hW}{\widehat{W}}
\DeclareMathOperator{\mmod}{mod}
\DeclareMathOperator{\dist}{dist}
\DeclareMathOperator{\diam}{diam}
\DeclareMathOperator{\HD}{HD}
\DeclareMathOperator{\hyp}{hyp}
\DeclareMathOperator{\HDhyp}{\HD_{\hyp}}

\newcommand{\cN}{\mathcal{N}}
\newcommand{\cP}{\mathcal{P}}
\newcommand{\cS}{\mathcal{S}}

\newcommand{\tB}{\widetilde{B}}

\newcommand{\C}{\mathbb{C}}

\newcommand{\eps}{\varepsilon}

\newcommand{\comSSSS}[1]{\marginpar{\tiny QQQ}}

\newcommand{\bfn}{\textbf{n}}
\newcommand{\borbit}{\mathcal{O}^-(0)}

\begin{document}

\title{On Poincar{\'e} series of unicritical polynomials at the critical point}

\author[J. Rivera-Letelier]{Juan Rivera-Letelier$^\dag$}
\address{Juan Rivera-Letelier, Facultad de Matem{\'a}ticas, Pontificia Universidad Cat{\'o}lica de Chile, Avenida Vicu{\~n}a Mackenna~4860, Santiago, Chile}
\email{riveraletelier@mat.puc.cl}
\thanks{$^\dag$ Partially supported by FONDECYT grant 1100922, Chile.}

\author[W. Shen]{Weixiao Shen$^\ddag$}
\address{Weixiao Shen, Block S17, 10 Lower Kent Ridge Road, Singapore~119076, Singapore}
\email{matsw@nus.edu.sg}
\thanks{$^\ddag$ Partially supported by National University of Singapore grant C-146-000-032-001}

\begin{abstract}
In this paper, we show that for a unicritical polynomial having \emph{a priori} bounds, the unique conformal measure of minimal exponent has no atom at the critical point.
For a conformal measure of higher exponent, we give a necessary and sufficient condition for the critical point to be an atom, in terms of the growth rate of the derivatives at the critical value.
\end{abstract}

\date{}

\maketitle

\section{Introduction}
Let $f:\C\to\C$ be a polynomial of degree at least~$2$, and let $J(f)$ denote its Julia set.
An important tool to study the fractal dimensions of the Julia set is the Patterson-Sullivan conformal measures.
Given~$t > 0$, a \emph{conformal measure of exponent~$t$ for~$f$} is a Borel probability measure~$\mu$ supported on~$J(f)$, such that for each Borel set~$E$ on which~$f$ is injective we have
$$ \mu(f(E)) = \int_E |Df|^t \ d\mu. $$
Sullivan showed that for every polynomial there is~$t > 0$ and a conformal measure of~$f$ of exponent~$t$, see~\cite{Sul83}.
Moreover, Denker, Przytycki, and Urba{\'n}ski showed that the minimal exponent for which such a measure exists coincides with the \emph{hyperbolic dimension of~$f$}, defined by
$$ \HDhyp(f) \= \{ \HD(X) : X \emph{ hyperbolic set of~$f$} \}, $$
where~$\HD(X)$ denotes the Hausdorff dimension of~$X$, see~\cite{DenUrb91c,Prz93}, and also~\cite{McM00,PrzUrb10,Urb03b}.
In the uniformly hyperbolic case, there is a unique conformal measure of minimal exponent, and this measure coincides with the normalized Hausdorff measure of dimension equal to~$\HD(J(f))$.
Under certain non-uniformly hyperbolicity assumptions, it has also been proved that there is a unique conformal measure of minimal exponent, and that this measure is supported on the ``conical Julia set'', which roughly speaking is the expanding part of the Julia set, see for example~\cite{Prz99,GraSmi09,PrzRiv07,RivShe1004,Urb03b} and references therein.
In these cases, the conformal measure of minimal exponent is used to prove that the hyperbolic dimension coincides with the Hausdorff and the box dimensions of the Julia set.
However, there are conformal measures of minimal exponent that are not supported on the conical Julia set, see~\cite{Urb03b}.

In this paper, we shall study atoms of conformal measures of polynomials having precisely~$1$ critical point; we call such a polynomial \emph{unicritical}.
Note that if~$f$ is a unicritical polynomial, then its degree~$d$ is at least~$2$, and there is~$c$ in~$\C$ such that~$f$ is affine conjugate to the polynomial~$z^d + c$.
A unicritical polynomial written in this form is \emph{normalized}.
We shall make the following assumption.

\begin{definition}
\label{def:apriorib}
Let~$f$ be a unicritical polynomial whose critical point is non-periodic and recurrent.
Assume for simplicity that~$f$ is normalized, so its critical point is~$0$.
Then we say~$f$ has \emph{a priori bounds}, if there exists $\tau > 0$ such that for each $\eps>0$ there exists a topological disk~$V$ containing~$0$, satisfying $\diam(V)<\eps$, and such that the following holds: For each integer~$n\ge 1$ such that~$f^n(0)$ is in~$V$, the connected component~$U$ of $f^{-n}(V)$ that contains~$0$ satisfies $\overline{U} \subset V$, and there is annulus~$A$ contained in~$V \setminus \overline{U}$, enclosing~$U$, and whose modulus is at least~$\tau$.
\end{definition}

Examples of unicritical polynomials having \emph{a priori} bounds include at most finitely renormalizable polynomials without neutral cycles~\cite{Hub93,KahLyu09b,KozvSt09}, all real polynomials~\cite{LevvSt98,LevvSt00a}, and some infinitely renormalizable complex polynomials~\cite{KahLyu08,Lyu97}.
Recall that for an integer~$s \ge 1$, a unicritical polynomial~$f$ of degree~$d \ge 2$ is \emph{renormalizable of period~$s$}, if there are Jordan disks~$U$ and~$V$ such that~$\overline{U}$ is contained in~$V$, and such that the following properties hold:
\begin{itemize}
\item
$f^s : U \to V$ is $d$-to-$1$;
\item
$U$ contains the critical point of~$f$, and for each~$j$ in~$\{1, \ldots, s - 1 \}$ the set~$f^j(U)$ does not contain it;
\item
The set~$\{z \in U : f^{sn}(z) \in U, n = 0, 1, \ldots \}$ is a connected proper subset of~$J(f)$.
\end{itemize}
Moreover, $f$ is \emph{infinitely renormalizable}, if there are infinitely many integers~$s$ such that~$f$ is renormalizable of period~$s$.

In the statement of the following theorem we use the fact that for a unicritical polynomial~$f$ having \emph{a priori} bounds there is a unique conformal measure of exponent~$t = \HDhyp(f)$ for~$f$, see~\cite[Theorem~$1$]{Pra98}.

\begin{theoalph}
\label{thm:main}
Let~$f$ be a unicritical polynomial having \emph{a priori} bounds.
Then the conformal measure~$\mu$ of minimal exponent of~$f$ does not have an atom at the critical point of~$f$.
\end{theoalph}

We remark that in this result, it is essential that we consider the conformal measure of minimal exponent, as opposed to a conformal measure of higher exponent.
In fact, every unicritical map~$f$ that is not uniformly hyperbolic has, for each~$t > \HDhyp(f)$, a conformal measure of exponent~$t$, see~\cite{GraSmi09} for the case~$f$ satisfies the Collet-Eckmann condition, and~\cite{PrzRivSmi03,PrzRivSmi04} for the case~$f$ does not.
In many cases, even if~$f$ has \emph{a priori} bounds, for each~$t > \HDhyp(f)$ there is conformal measure of exponent~$t$ that is supported on the backward orbit of the critical point.

The existence of an atom at a point~$w$ for a conformal measure of~$f$ of exponent~$t$, is closely related to convergence of the following series:
$$ \cP(w, t)
\=
\sum_{n=0}^\infty\sum_{z\in f^{-n}(w)} |Df^n(z)|^{-t}; $$
it is the \emph{Poincar{\'e} series at~$w$ of exponent~$t$}.
In fact:
\begin{itemize}
\item If a conformal measure of exponent~$t$ has an atom at~$w$, then $\cP(w, t)<\infty$;
\item A conformal measure of exponent~$t$ has an atom at the critical point~$c_0$ if and only if $\cP(c_0, t)<\infty$.
\end{itemize}

\begin{theoalph}
\label{thm:reduced}
Let~$f$ be a unicritical polynomial having \emph{a priori} bounds.
Assume~$f$ is normalized so its critical point is~$0$.
Then for each $t > \HDhyp(f)$, the series~$\cP(0, t)$ is finite, if and only if~$\sum_{n=0}^\infty |Df^n(f(0))|^{-t/d}$ is finite.
\end{theoalph}

\subsection{Strategy and organization}
\label{ss:organization}
We derive Theorem~\ref{thm:main} from Theorem~\ref{thm:reduced} arguing by contradiction: If the conformal measure of minimal exponent had an atom at the critical point, by Theorem~\ref{thm:reduced} the derivatives at the critical value would grow to infinity.
Together with the \emph{a priori} bounds hypothesis, this implies that the map is ``backward contracting'' (Theorem~\ref{thm:lb2bc}), so the results of~\cite{RivShe1004} apply to this map; in particular, the Poincar{\'e} series at the critical point diverges.
This contradicts the existence of an atom at the critical point.
In~\S\ref{s:preliminaries} we fix some notation and terminology and recall the definition of backward contracting maps.
We deduce Theorem~\ref{thm:main} from Theorem~\ref{thm:reduced} in~\S\ref{s:a priori bounds and bc}, after proving Theorem~\ref{thm:lb2bc}.

To prove Theorem~\ref{thm:reduced}, we show that in either case the map is backward contracting (part~$1$ of Theorem~\ref{thm:lb2bc}).
This allows us to use the characterization for backward contracting maps of the summability condition given in~\cite{LiSheIMRN}.
We recall this result in~\S\ref{s:a priori bounds and bc}, as part~$2$ of Theorem~\ref{thm:lb2bc}.
To prove the direct implication in Theorem~\ref{thm:reduced}, we divide the integral of the backward contraction function, characterizing the summability condition, into intervals bounded by consecutive close return scales.
Then we show that each of these integrals is bounded by one of the terms in the Poincar{\'e} series up to a multiplicative constant (Lemma~\ref{l:close return integral}).
This is done in~\S\ref{s:direct implication of reduced}.
The proof of the reverse implication in Theorem~\ref{thm:reduced} occupies~\S\ref{s:reverse implication of reduced} and is more involved.
We use a discretized sequence of scales to code each iterated preimage of the critical point.
To do this, we consider the largest scale whose pull-back is conformal, and consider the ``critical hits'' when pulling back the previous scale, as a code.
One of the crucial estimates is the contribution in the Poincar{\'e} series of those iterated preimages of the critical point for which a certain ball can be pulled back conformally (Lemma~\ref{l:conformal Poincare}).
This estimate is done using one of the main results of~\cite{RivShe1004}: For a backward contracting map the diameter of pull-backs decreases at a super-polynomial rate.

\section{Preliminaries}
\label{s:preliminaries}
We endow~$\C$ with its norm distance, and for a bounded subset~$W$ of~$\C$ we denote by~$\diam(W)$ the diameter of~$W$.
Moreover, for~$\delta > 0$ and for a point~$z$ of~$\C$, we denote by~$B(z, \delta)$ the ball of~$\C$ centered at~$z$ and of radius~$\delta$.

A \emph{topological disk} is an open subset of~$\C$ homeomorphic to the unit disk, and that is not equal to~$\C$.
We endow such a set with its hyperbolic metric.
If~$V$ and~$V'$ are topological disks such that~$\overline{V'} \subset V$, we denote by~$\mmod(V; V')$ the supremum of the moduli of all annuli contained in~$V \setminus V'$ that enclose~$V'$.
See for example~\cite{Mil06,CarGam93} for background.

Throughout the rest of this paper we fix an integer~$d \ge 2$, a parameter~$c$ in~$\C$, and put~$f(z) \= z^d + c$.
Moreover, we assume~$0$ is non-periodic and recurrent by~$f$; this implies that~$0$ is in the Julia set of~$J(f)$.
Given a subset~$V$ of~$\C$, and an integer~$n \ge 0$, each connected component of~$f^{-n}(V)$ is called a \emph{a pull-back of~$V$ by~$f^n$}.
When~$n \ge 1$, a pull-back of~$V$ is \emph{critical} if it contains~$0$.

\subsection{Backward contraction}
\label{ss:backward contraction}
In this section we recall the definition of ``backward contraction'' from~\cite{Riv07}, and compare it with a variant from~\cite{LiSheIMRN}.

For each~$\delta > 0$, put $\tB(\delta) \= B(0, \delta^{1/d})$, and
$$ r(\delta)
\=
\sup \left\{ r > 0 :
\begin{array}{ll}
\text{ for every pull-back~$U$ of~$\tB(\delta r)$,}
\\
\text{ $\dist(W, c) \le \delta$ implies~$\diam(W) < \delta$}
\end{array}
\right\}. $$
Given~$r_0 > 1$, the map~$f$ is \emph{backward contracting with constant~$r_0$}, if for every sufficiently small~$\delta$ we have~$r(\delta) > r_0$.
Moreover, $f$ is \emph{backward contracting} if $r(\delta) \to \infty$ as~$\delta \to 0$.

For each~$\delta > 0$, put
$$ R(\delta)
\=
\inf\left\{
\frac{\delta}{\diam(U)}:
U \text{ pull-back of $\tB(\delta)$ containing~$c$} \right\}.$$
Clearly, for every~$\delta > 0$ and~$\delta' > \delta$ we have~$R(\delta) \ge (\delta / \delta') R(\delta')$.
Moreover, for every~$\delta > 0$ and every~$r$ in~$(0, r(\delta))$, we have~$R(\delta r) \ge r$.
So, if~$f$ is backward contracting with some constant~$r_0 > 1$, then for every sufficiently small~$\delta$ we have~$R(\delta) \ge r_0$.
In particular, if~$f$ is backward contracting, then~$R(\delta) \to \infty$ as~$\delta \to 0$.

\begin{lem}
\label{l:backward contraction}
There is a constant~$\rho_0 > 1$ independent of~$f$ and~$d$, such that for every~$\delta > 0$ satisfying~$R(\delta) \ge \rho_0^d$, and every~$\delta'$ in~$[\delta / R(\delta), \delta)$, we have~$r(\delta') \ge \rho_0^{-d} \delta / \delta'$.
In particular, if there is~$R_0 \ge \rho_0^d$ such that for every sufficiently small~$\delta > 0$ we have~$R(\delta) > R_0$, then~$f$ is backward contracting with constant~$\rho_0^{-d} R_0$.
Moreover, $f$ is backward contracting if and only if~$R(\delta) \to \infty$ as~$\delta \to \infty$.
\end{lem}
\begin{proof}
Let~$\rho_0 > 1$ be sufficiently large so that for every pair of topological disks~$U$ and~$V$ such that
$$ \overline{U} \subset V
\text{ and }
\mmod(V; U) \ge \log \rho_0, $$
we have~$\diam(U) < \dist(U, \partial V)$.

Fix~$\delta > 0$ such that~$R(\delta) \ge \rho_0^d$, and~$\delta'$ in~$[\delta / R(\delta), \delta)$.
To prove the lemma it suffices to show that for every integer~$n \ge 0$, and every pull-back~$W$ of~$\tB(\rho_0^{- d} \delta)$ by~$f^n$ satisfying~$\dist(W, c) \le \delta'$, we have~$\diam(W) < \delta'$.
We proceed by induction in~$n$.
When~$n = 0$, we have~$W = \tB(\rho_0^{-d} \delta)$.
If~$c$ is in~$\tB(\delta)$, then by the definition of~$R(\delta)$ we have
$$ \diam(W)
<
\diam(\tB(\delta))
\le
\delta / R(\delta)
\le
\delta'. $$
Otherwise, by the definition of~$\rho_0$ we have
$$ \diam(W)
<
\dist(W, \partial \tB(\delta))
\le
\dist(W, c)
\le
\delta'. $$
This proves the desired assertion when~$n = 0$.
Let~$n \ge 1$ be an integer such that the desired assertion holds with~$n$ replaced by each element of~$\{0, \ldots, n - 1 \}$.
Let~$W$ be a pull-back of~$\tB(\rho_0^{- d} \delta)$ by~$f^n$ satisfying~$\dist(W, c) \le \delta'$.
If the pull-back~$\hW$ of~$\tB(\delta)$ containing~$W$ contains~$c$, then
$$\mmod(\hW;W)=\mmod(\tB(\delta);\tB(\rho_0^{-d} \delta))=\log \rho_0$$
so by the definition of~$R(\delta)$ we have
$$ \diam(W)
<
\diam(\hW)
\le
\delta / R(\delta)
\le
\delta'. $$
So we suppose~$\hW$ does not contain~$c$.
If~$f^n$ is conformal on~$\hW$, then by the definition of~$\rho_0$ we have
$$ \diam(W) < \dist(W, \partial \hW) \le \dist(W, c) \le \delta'. $$
It remains to consider the case where~$f^n$ is not conformal on~$\hW$, so there is~$m$ in~$\{0, \ldots, n - 1 \}$ such that~$f^m(\hW)$ contains~$0$.
Then~$f^{m + 1}(\hW)$ contains~$c$, and by definition of~$R(\delta)$ we have
$$ \diam(f^{m + 1}(\hW))
\le
\delta / R(\delta)
\le
\rho_0^{-d} \delta. $$
This implies that~$f^{m + 1}(\hW)$ is contained in~$B(c, \rho_0^{-d} \delta)$, and therefore that~$f^m(\hW)$ is contained in~$\tB(\rho_0^{-d} \delta)$.
Using the induction hypothesis, we conclude that~$\diam(W) < \delta'$.
This completes the proof of the induction step and of the lemma.
\end{proof}

\subsection{Nice sets and children}
\label{ss:nice sets and children}
The purpose of this section is to prove a general lemma about backward contracting maps that is used in~\S\ref{ss:thickened grandchildren} to prove Theorem~\ref{thm:reduced}.

A topological disk~$V$ containing~$0$ is a \emph{nice set for~$f$}
if for every integer~$n \ge 1$ the set~$f^n(\partial V)$ is disjoint from~$V$.

For an integer~$m \ge 1$, and a connected open set~$V$, a pull-back~$W$ of~$V$ by~$f^m$ is a \emph{child of~$V$}, if~$f^m$ maps~$W$ onto~$V$ as a $d$-to-$1$ map and~$W\ni 0$.

The following lemma is a more precise version of~\cite[Lemma~$3.15$]{RivShe1004}, with the same proof.

\begin{lem}
\label{l:child contribution}
Suppose~$f$ is backward contracting.
Then for every~$s > 0$ there is~$\delta_0 > 0$, such that for every~$\delta$ in~$(0, \delta_0]$ there is a nice set~$V$ for~$f$ satisfying~$\tB(\delta/2) \subset V \subset \tB(\delta)$, and such that
$$ \sum_{Y : \text{ child of } V} \diam(f(Y))^s
\le
(1 - 2^{-s})^{-1} \left( R(\delta)^{-1} \delta \right)^s. $$
\end{lem}
\begin{proof}
Let~$\lambda > 0$ be sufficiently large so that for every pair of topological discs~$Y$ and~$Y'$ satisfying
$$ \overline{Y'} \subset Y
\text{ and }
\mmod(Y; Y') \ge \lambda / d, $$
we have~$\diam(Y') \le \diam(Y) / 2$, see for example~\cite[Theorem~$2.1$]{McM94}.
In view of~\cite[Lemma~$3.13$]{RivShe1004}, for every sufficiently small~$\delta > 0$ there is a $\lambda$\nobreakdash-nice set~$V$ for~$f$ such that~$\tB(\delta/2) \subset V \subset \tB(\delta)$.
We prove the desired assertion hold for this choice of~$V$.

For each integer~$k \ge 1$, let~$Y_k$ be the $k$-th smallest child of~$V$ and let~$s_k$ be the integer such that~$f^{s_k}(Y_k) = V$.
By the backward contracting property, we have $\diam(f(Y_1)) \le R(\delta)^{-1}\delta$.
Note that for each~$k \ge 1$ the set~$f^{s_k}(Y_{k+1})$ is contained in a return domain of~$V$, so~$\mmod(V; f^{s_k}(Y_{k+1})) \ge \lambda$.
By the definition of child, the map~$f^{s_k} : Y_k \to V$ is $d$-to-$1$, thus~$\mmod(Y_k; Y_{k+1}) \ge \lambda/d$ and therefore~$\diam(Y_{k + 1}) \le \diam(Y_k)/2$.
The conclusion of the lemma follows.
\end{proof}

\section{\emph{A priori} bounds and backward contraction}
\label{s:a priori bounds and bc}
In this section we derive Theorem~\ref{thm:main} from Theorem~\ref{thm:reduced}.
To do so, we first establish a sufficient criterion for a unicritical map having \emph{a priori} bounds to be backward contracting (Theorem~\ref{thm:lb2bc} below).

In the following theorem we summarize and complement results in~\cite{LiShe10b,LiSheIMRN}, when restricted to unicritical maps.
We state it in a stronger form than what is needed for this section.
\begin{thm}
\label{thm:lb2bc}
For each~$\tau > 0$ there is a constant~$\eta > 1$, such that if~$f$ has \emph{a priori} bounds with constant~$\tau$, then the following properties hold.
\begin{enumerate}
\item[1.]
Let~$R_0 > 1$ be such that for every sufficiently large~$n$ we have either
$$|Df^n(c)| \ge \eta R_0,
\text{ or }
\min_{\zeta\in f^{-n}(0)}|Df^n(\zeta)| \ge (\eta R_0)^{1/d}. $$
Then for every sufficiently small~$\delta > 0$ we have~$R(\delta) \ge R_0$.
\item[2.]
For each~$t > 0$, the sum~$\sum_{n = 0}^{\infty} |Df^n(c)|^{- t}$ is finite if and only if~$\int_{0+}^{\infty} R(\delta)^{-t} \ \frac{d\delta}{\delta}$ is finite.
\end{enumerate}
\end{thm}
When~$|Df^n(c)|$ is eventually bounded from below by~$\eta R_0$, part~$1$ is given by (the proof of)~\cite[Theorem~A]{LiShe10b}.
Part~$2$ is a direct consequence of part~$1$ and~\cite[Theorems~$1.3$ and~$1.4$]{LiSheIMRN}, together with Lemma~\ref{l:backward contraction} and~\cite[Corollary~$8.3$]{Riv07}.

To prove this theorem, we shall use the following variation of the Koebe distortion theorem.

\begin{lem}
\label{lem:koebe}
For each~$\tau_0 > 0$ there is a constant~$C_{@} > 1$, such that for every pair of topological disks~$V$ and~$V'$ containing~$0$, and satisfying
$$ \overline{V'} \subset V
\text{ and }
\mmod(V;V') \ge \tau_0, $$
the following property holds.
Let~$s \ge 1$ be a integer such that~$f^s(0)$ is in~$V'$, and such that~$f^{s - 1}$ maps a neighborhood~$W$ of~$c$ conformally onto~$V$.
Moreover, let~$W'$ be the pull-back of~$V'$ by~$f^{s - 1}$ contained in~$W$, and let~$\zeta$ in~$f^{-s}(0)$ be such that~$f(\zeta)$ is in~$W'$.
Then we have
$$ \diam(W')
\le
C_{@} \max \left\{ |Df^s(c)|, |Df^s(\zeta)|^d \right\}^{-1} \diam (f(V')). $$
\end{lem}
\begin{proof}
By Koebe distortion theorem there is a constant~$K_1 > 1$ that only depends on~$\tau_0$, such that the distortion of~$f^{s - 1}$ on~$W'$ is bounded by~$K_1$.
We thus have,
\begin{multline*}
\diam(W')
\le
K_1 |Df^{s - 1}(c)|^{-1} \diam(V')
\\ =
K_1 |Df^s(c)|^{-1} |Df(f^{s - 1}(c))| \diam(V').
\end{multline*}
Since~$|Df(f^{s - 1}(c))| = d|f^{s - 1}(c)|^{d - 1} \le d \diam(V')^{d - 1}$, we obtain
\begin{equation}
  \label{e:first half}
\diam(W')
\le
d K_1 |Df^s(c)|^{-1} \diam(V')^d
\le
d 2^d K_1 |Df^s(c)|^{-1} \diam(f(V')).
\end{equation}

On the other hand, we have~$|Df^{s - 1}(f(\zeta))| \le K_1 |Df^{s - 1}(c)|$.
Using the formula of~$f$, we obtain
\begin{equation}
\label{e:asymmetric comparison}
|Df^s(\zeta)|
\le
K_1 (|\zeta| / |f^{s - 1}(c)|)^{d - 1} |Df^s(c)|.
\end{equation}
Using Koebe distortion theorem and the formula of~$f$ again, we obtain
$$ |f^{s - 1}(c)|
\ge
K_2^{-1} |Df^{s - 1}(f(\zeta))| \cdot |f(\zeta) - c|
=
(d K_2)^{-1} |Df^s(\zeta)| \cdot |\zeta|, $$
where $K_2>1$ is a constant depending only on $\tau_0$.
Together with~\eqref{e:asymmetric comparison}, we obtain~$|Df^s(\zeta)|^d \le d^{d - 1} K^d |Df^s(c)|$,
where $K=\max \{ K_1, K_2 \}$.
Combined with~\eqref{e:first half}, this implies the desired inequality with~$C_{@} = (2d)^d K^{d + 1}$.
\end{proof}

For each topological disk~$V$ containing~$0$, put
$$ M_+(V)
\=
\inf \left\{ |Df^n(c)| : n \ge 1, f^n(0)\in V \right\}, $$
$$ M_-(V)
\=
\inf \left\{ |Df^n(\zeta)|^d : n \ge 1, \zeta \in V, f^n(\zeta) = 0 \right\},$$
and
$$ M(V)
\=
\max \{ M_+(V), M_-(V) \}. $$
\begin{lem}
\label{l:diameter derivative estimate}
For every constant~$\tau_0 > 0$ there is a constant~$C_{!} > 0$ such that the following property holds.
Let~$V$ and~$V'$ be topological disks containing~$0$, such that~$\overline{V'} \subset V$, $\mmod(V; V') \ge \tau_0$, and such that every critical pull-back of~$V$ is contained in~$V'$.
Then for every critical pull-back~$U'$ of~$V'$, we have
$$ \diam(f(U'))
\le
C_{!} M(V)^{-1} \diam(f(V')). $$
\end{lem}
\begin{proof}
Let~$C_{@}$ be the constant given by Lemma~\ref{lem:koebe}.
Let~$n \ge 1$ be an integer such that~$f^n(0)$ is in~$V'$, and let~$U$ (resp.~$U'$) be the pull-back of~$V$ (resp.~$V'$) by~$f^n$ containing~$0$.
It suffices to consider the case that~$U$ is not contained in any critical pull-back of~$V'$.
In this case we claim that~$f^n: U \to V$ is $d$-to-$1$.
Otherwise, there would exist~$m$ in~$\{1, \ldots, n - 1 \}$ such that~$f^m(U)$ contains~$0$.
This implies that~$f^m(U)$ is contained in the pull-back of~$V$ by~$f^{n - m}$ containing~$0$, which is contained in~$V'$.
Thus~$f^m(U)$ is contained in~$V'$, and therefore~$U$ is contained in the pull-back of~$V'$ by~$f^m$ containing~$0$.
We thus obtain a contradiction that shows that~$f^n : U \to V$ is $d$-to-$1$.
Then the lemma follows from Lemma~\ref{lem:koebe} with~$C_{!} = C_{@}$.
\end{proof}

Given a topological disk~$V$, a point~$x$ of~$V$, and~$\Delta > 0$, denote by~$B_V(x, \Delta)$ the ball for the hyperbolic metric of~$V$ centered at~$x$ and of radius~$\Delta$.
By Koebe distortion theorem, for every~$\Delta > 0$ there is~$\xi > 1$ such that for every topological disk~$V$, and every~$x$ in~$V$, we have
$$ \sup_{z \in \partial B_V(x, \Delta)} |z-x|
\le
\xi
\inf_{z \in \partial B_V(x, \Delta)} |z-x|. $$
\begin{proof}[Proof of Theorem~\ref{thm:lb2bc}]
As mentioned above, part~$2$ is a direct consequence of part~$1$, and the combination of~\cite[Theorems~$1.3$ and~$1.4$]{LiSheIMRN}, Lemma~\ref{l:backward contraction}, and~\cite[Corollary~$8.3$]{Riv07}.

To prove part~$1$, note that for~$\delta > 0$ the number~$\mmod(\tB(\delta); \tB(\delta/2))$ is independent of~$\delta$; denote it by~$\tau_1$.
On the other hand, let~$\tau$ be the constant given by the \emph{a priori} bounds hypothesis.
Note that there is a constant~$\Delta > 0$ such that if~$U$ is a topological disk satisfying~$\overline{U} \subset V$ and~$\mmod(V; U) \ge \tau$, then the diameter of~$U$ with respect to the hyperbolic metric of~$V$ is bounded by~$\Delta$.
Moreover, the number~$\mmod ( V; B_V(0, \Delta))$ is bounded from below by a constant~$\tau_0>0$ that is independent of~$V$.
Let~$C_{!}$ be the constant given by Lemma~\ref{l:diameter derivative estimate} with~$\tau_0$ replaced by~$\min\{ \tau_0, \tau_1 \}$.
Let~$\xi > 1$ be the constant defined above for this choice of~$\Delta$, and put~$\eta \= 4 C_{!} \xi^d$.

Let~$\varepsilon > 0$ be sufficiently small so that for every topological disk~$V$ containing~$0$ and satisfying~$\diam(V) < \varepsilon$, we have~$M(V) \ge \eta R_0$.
By the \emph{a priori} bounds hypothesis, there is such a topological disk so that in addition for every critical pull-back~$U$ of~$V$ we have~$\overline{U} \subset V$ and~$\mmod(V; U) \ge \tau$.
By our choice of~$\Delta$, this implies that~$U$ is contained in~$V' \= B_V(0, \Delta)$.
So the hypotheses of Lemma~\ref{l:diameter derivative estimate} are satisfied for these choices of~$V$, and~$V'$.
It follows that for every critical pull-back~$U'$ of~$V'$, we have
$$ \diam(f(U'))
\le
C_{!} M(V)^{-1} \diam(f(V'))
\le
4^{-1} \xi^{-d} R_0^{-1} \diam(f(V')). $$
Thus, if we put $\delta_0 \= \inf_{z \in \partial V'} |z|^d$, then~$\diam(f(V')) \le 2 \xi^d \delta_0$, and for every critical pull-back~$U''$ of~$\tB(\delta_0)$ we have
$$ \diam(f(U''))
\le
(2 R_0)^{-1} \delta_0. $$
This proves~$R(\delta_0) \ge 2 R_0$.

To complete the proof part~$1$, for each integer~$n \ge 1$ put~$\delta_n \= 2^{-n} \delta_0$.
We prove by induction that for each~$n$ we have~$R(\delta_n) \ge 2 R_0$.
The case~$n = 0$ is shown above.
Let~$n \ge 1$ be an integer such that~$R(\delta_{n - 1}) \ge 2R_0 > 2$.
Then every critical pull-back~$\tB(\delta_{n - 1})$ is contained in~$\tB(\delta_n)$.
So by Lemma~\ref{l:diameter derivative estimate}, for every critical pull-back~$U'$ of~$\tB(\delta_n)$ we have
$$ \diam(f(U'))
\le
2 C_{!} M(\tB(\delta_{n - 1}))^{-1} \delta_n
\le
(2 R_0)^{-1} \delta_n. $$
This proves~$R(\delta_n) \ge 2R_0$ and completes the proof of the induction step.
Thus for every~$n \ge 0$ we have~$R(\delta_n) \ge 2 R_0$, and therefore for every~$\delta$ in~$(0, \delta_0)$ we have~$R(\delta) \ge R_0$.
This completes the proof of part~$1$ and of the theorem.
\end{proof}

\begin{proof}[Proof of Theorem~\ref{thm:main} assuming Theorem~\ref{thm:reduced}]
Let~$\mu$ denote the conformal measure of minimal exponent~$h=\HDhyp(f)$.
Assume by contraction that $\mu$ has an atom at $0$.
Then $\cP(0, h)<\infty$, hence for each $t>h$ we have $\cP(0, t)<\infty$.
By Theorem~\ref{thm:reduced}, it follows that $\sum_{n=0}^\infty |Df^n(c)|^{-t/d}<\infty$, hence $|Df^n(c)|\to\infty$ as $n\to\infty$.
By Lemma~\ref{l:backward contraction} and part~$1$ of Theorem~\ref{thm:lb2bc}, the map $f$ is backward contracting.
But then, \cite[Proposition~$7.3$]{RivShe1004} implies~$\cP(0, h)=\infty$.
We thus obtain a contradiction that completes the proof of the theorem.
\end{proof}
\section{Convergence of Poincar{\'e} series implies forward summability}
\label{s:direct implication of reduced}
The purpose of this section is to prove the following proposition, giving one of the implications in Theorem~\ref{thm:reduced}.

\begin{prop}
\label{p:direct implication}
Suppose~$f$ satisfies the \emph{a priori} bounds condition.
Then for each~$t > 0$ such that~$\cP(0, t)$ is finite, the sum $\sum_{n=1}^\infty |Df^n(c)|^{-t/d}$ is also finite.
\end{prop}

For each $\delta>0$, let $n(\delta)$ be the minimal integer~$n \ge 1$ such that~$f^n(0)$ is in~$\tB(\delta)$, let~$U_\delta$ denote the pull-back of~$\tB(\delta)$ by~$f^{n(\delta)}$ that contains~$0$, and let~$\zeta(\delta)$ be a point of~$f^{- n(\delta)}(0)$ in~$U_\delta$.
Clearly, $n(\delta)$ is non-increasing with~$\delta$, left continuous, and we have $n(\delta) \to \infty$ as~$\delta \to 0$.
In view of Theorem~\ref{thm:lb2bc}, Proposition~\ref{p:direct implication} is a direct consequence of the following lemma.

\begin{lem}
\label{l:close return integral}
There is~$\delta_0 > 0$ such that for each~$t > 0$ there is~$C_{\dag} > 0$ such that the following property holds: For every~$\delta_1$ and~$\delta_2$ in~$(0, \delta_0)$ satisfying~$\delta_1 < \delta_2$, and such that~$n(\cdot)$ is constant on~$(\delta_1, \delta_2]$, we have
$$ \int_{\delta_1}^{\delta_2} R(\delta)^{-t/d} \ \frac{d\delta}{\delta}
\le
C_{\dag} |Df^{n(\delta_2)}(\zeta(\delta_2))|^{-t}.$$
\end{lem}

The proof of this lemma is after the following one.

\begin{lem}
\label{l:non-critical bc}
Assume that~$f$ is backward contracting, and for each~$\delta > 0$ let~$W_\delta$ be the pull-back of~$\tB(2\delta)$ by~$f^{n(\delta) - 1}$ containing~$c$.
Then for every sufficiently small~$\delta$ we have~$R(\delta) \ge \delta / \diam(W_\delta)$.
\end{lem}

\begin{proof}
Let~$\delta_0 > 0$ be sufficiently small so that for every~$\delta$ in~$(0, \delta_0)$ we have~$r(\delta) > 2$.
It suffices to show that for each~$\delta$ in~$(0, \delta_0)$, and every integer~$m \ge 0$ such that~$f^m(c)$ is in~$\tB(\delta)$, the pull-back~$W$ of~$\tB(\delta)$ by~$f^m$ containing~$c$ is contained in~$W_{\delta}$.
Clearly~$m \ge n(\delta) - 1$, and when~$m = n(\delta) - 1$ we have~$W \subset W_{\delta}$.
If~$m \ge n(\delta)$, then our hypothesis~$r(\delta) > 2$ implies that~$f^{n(\delta) - 1}(W)$ is contained in~$\tB(2 \delta)$.
This shows that~$W$ is contained in~$W_{\delta}$, as wanted.
\end{proof}

\begin{proof}[Proof of Lemma~\ref{l:close return integral}]
By Koebe Distortion Theorem there is a constant~$K > 1$ such that for every~$\delta > 0$, every integer~$n \ge 1$, and every pull-back~$W$ of~$\tB(2\delta)$ by~$f^n$ for which the corresponding pull-back of~$\tB(4 \delta)$ is conformal, the distortion of~$f^n$ on~$W$ is bounded by~$K$.
Let~$\delta_0 > 0$ be such that the conclusion of Lemma~\ref{l:non-critical bc} holds for every~$\delta$ in~$(0, \delta_0)$.
Reducing~$\delta_0$ if necessary, assume that for each~$\delta$ in~$(0, 4\delta_0)$ we have~$R(4 \delta) \ge 4$.

To prove the lemma, let~$\delta_1$ and~$\delta_2$ in~$(0, \delta_0)$ be such that~$n(\cdot)$ is constant on~$(\delta_1, \delta_2]$.
Put
$$ n \= n(\delta_2)
\text{ and }
\zeta \= \zeta(\delta_2), $$
and note that~$|f^{n - 1}(c)| = |f^n(0)| \le \delta_1^{1/d}$.
Let~$\delta$ in~$(\delta_1, \delta_2]$, and let~$W_{\delta}$ (resp.~$\hW_{\delta}$) be the pull-back of~$\tB(2 \delta)$ (resp.~$\tB(4 \delta)$) by~$f^{n - 1}$ containing~$c$.
We claim that~$f^{n - 1}$ is conformal on~$\hW_{\delta}$.
Otherwise, we would have~$n \ge 2$, and there would exist~$m$ in~$\{1, \ldots, n - 1 \}$ such that~$f^m(\hW_{\delta})$ contains~$c$.
However, this implies that
$$ \diam(f^m(\hW_{\delta})) \le 4 \delta / R(4 \delta) \le \delta, $$
and therefore that~$f^{m - 1}(\hW_{\delta})$, and hence~$f^{m - 1}(0)$, is contained in~$\tB(\delta)$; we thus obtain a contradiction with the definition of~$n = n(\delta)$ that proves that~$f^{n - 1}$ is conformal on~$\hW_{\delta}$.

Using Koebe distortion theorem and the formula of~$f$, we have
\begin{equation}
\label{l:first entry pull-back}
\diam(W_{\delta})
\le
K (2 \delta)^{1/ d} |Df^{n - 1}(f(\zeta))|^{-1}
=
d K (2 \delta)^{1/ d} |Df^n(\zeta)|^{-1} |\zeta|^{d - 1},
\end{equation}
and
\begin{multline*}
|\zeta|^d
=
|f(\zeta) - c|
\le
K |Df^{n - 1}(f(\zeta))|^{-1} |f^{n - 1}(c)|
\\ \le
K \delta_1^{1/d} |Df^{n - 1}(f(\zeta))|^{-1}
=
d K \delta_1^{1/d} |Df^n(\zeta)|^{-1} |\zeta|^{d - 1},
\end{multline*}
and therefore~$|\zeta| \le d K \delta_1^{1/d} |Df^n(\zeta)|^{-1}$.
Hence, letting~$C_{\ddag} \= 2^{1/d} (d K)^d$, by~\eqref{l:first entry pull-back} we have
$$ \diam(W_{\delta})
\le
C_{\ddag} \delta_1 (\delta / \delta_1)^{1/d} |D f^n(\zeta)|^{-d}. $$
Together with Lemma~\ref{l:non-critical bc}, this implies
$$ R(\delta) \ge C_{\ddag}^{-1} (\delta / \delta_1)^{1 - 1/d} |D f^n(\zeta)|^d. $$
The desired assertion follows with~$C_{\dag} = C_{\ddag}^{t/d} \int_1^{\infty} \eta^{- 1 - (t/d)(1 - 1/d)} \ d \eta$.
\end{proof}

\section{Forward summability implies backward summability}
\label{s:reverse implication of reduced}
In this section we complete the proof of Theorem~\ref{thm:reduced}.
After some estimates in~\S\ref{ss:thickened grandchildren}, the proof of Theorem~\ref{thm:reduced} is in~\S\ref{ss:proof of reduced}.

\subsection{Thickened grandchildren}
\label{ss:thickened grandchildren}
Throughout this section, assume~$f$ is backward contracting, so that~$R(\delta) \to \infty$ as~$\delta \to 0$, see~\S\ref{ss:backward contraction}.
Moreover, fix~$t > 0$, put~$\tau := 2^{-d}$, and let~$\delta_0 > 0$ be given by Lemma~\ref{l:child contribution} with~$s = t / d$.
For every integer~$q \ge 0$, let~$V_q$ be given by Lemma~\ref{l:child contribution} with~$\delta = \tau^q \delta_0$, so that
$$ \tB(\tau^q \delta_0/2) \subset V_q \subset \tB(\tau^q \delta_0). $$
Note that for every~$q \ge 0$ the set~$\overline{V_{q + 1}}$ is contained in~$V_q$ and~$\mmod(V_q; V_{q + 1})$ is bounded from below independently of~$q$.
Reducing~$\delta_0$ if necessary, assume that for every~$\delta$ in~$(0, \delta_0]$ we have~$R(\delta) > 2\tau^{-1}$.

For each integer~$q \ge 0$, let~$\cN(q)$ be the set of all integers~$n \ge 1$ such that~$f^n(0)$ is in~$V_q$, and such that the pull-back~$W_n(q)$ of~$V_q$ by~$f^n$ containing~$0$ is a child of~$V_q$.
By our choice of~$\delta_0$, the set~$W_n(q)$ is contained in~$\tB(\tau^{q + 1}\delta_0/2)$, and hence in~$V_{q + 1}$; denote by~$p_n(q)$ the largest integer~$p \ge q$, such that~$W_n(q)$ is contained in~$V_{p + 1}$.
Then we have
$$ \diam(f(W_n(q))) \ge \tau^{p_n(q) + 2} \delta_0/2. $$
Combined with Lemma~\ref{l:child contribution}, this implies
\begin{equation}
  \label{e:discretized child contribution}
  \sum_{n \in \cN(q)} \tau^{t \cdot p_n(q) /d}
\le
(1 - 2^{- t/d})^{-1} \left( 2 R(\tau^q \delta_0)^{-1} \tau^{q - 2} \right)^{t /d}.
\end{equation}

Given an integer~$q \ge 0$, let~$\cS(q)$ be the set of all finite sequences of integers~$(n_1, \ldots, n_k)$, such that there are integers
$$ p_0 := q, p_1, \ldots, p_k, $$
satisfying the following property: For each~$i$ in~$\{1, \ldots, k\}$ the integer~$n_i$ is in~$\cN(p_{i - 1})$, and~$p_i = p_{n_i}(p_{i - 1})$.
Note that in this case for each~$i$ in~$\{1, \ldots, k \}$ we have~$p_i \ge p_{i - 1} + 1$.
In the situation above, we put~$p_{(n_1, \ldots, n_k)}(q) := p_k$.

Let~$q \ge 0$ be an integer.
For each $k \ge 1$ and~$\bfn = (n_1, \ldots, n_k)$ in~$\cS(q)$, put
$$ |\bfn| := n_1 + \cdots + n_k
\text{ and }
k(\bfn) := k. $$
Moreover, let~$Z_{\bfn}(q)$ be the set of all points~$z$ of~$f^{-|\bfn|}(0)$ such that~$f^{|\bfn|}$ maps a neighborhood of~$z$ conformally onto~$V_{q + 1}$, and such that in the case~$k \ge 2$, for each~$i$ in~$\{2, \ldots, k \}$ the point~$f^{n_k + \cdots + n_{i + 1}}(z)$ is in the pull-back of~$V_{p_{(n_1, \ldots, n_{i - 1})}}$ by~$f^{n_i}$ containing~$0$.
Note that~$\# Z_{\bfn}(q) \le d^{k(\bfn)}$, and that~$Z_{\bfn}(q)$ is contained in~$V_{p_{\bfn}(q) + 1}$.

The purpose of this section is to prove the following.
\begin{prop}
\label{prop:poinbad}
Suppose~$f$ is backward contracting, and put
$$ C_{\#} := 2^{1 + 2t/d} d (1 - 2^{- t/d})^{-1} \tau^{- 2 t / d}. $$
Then, for all~$q \ge 0$ sufficiently large we have
$$ \sum_{\bfn \in \cS(q)} \sum_{z \in Z_{\bfn}(q)} \left| Df^{|\bfn|}(z) \right|^{-t}
\le
C_{\#} R(\tau^q \delta_0)^{-t/d}. $$
\end{prop}
The proof of this proposition is given after the following lemma.
\begin{lem}
\label{l:contribution of a bad chain}
Suppose~$f$ is backward contracting.
Then, for every integer~$q \ge 0$ and every~$\bfn$ in~$\cS(q)$ we have
$$ \sum_{z\in Z_{\bfn}(q)} \left| Df^{|\bfn|}(z) \right|^{-t}
\le
2^{t/d} d^{k(\bfn)} \left(\tau^{p_{\bfn}(q) - q} \right)^{t/d}. $$
\end{lem}
\begin{proof}
By definition, for each~$z$ in~$Z_{\bfn}(q)$ the map~$f^{|\bfn|}$ maps a neighborhood~$U$ of~$z$ conformally onto~$V_{q+1}$.
Noting that~$U$ is contained in~$V_{p_{\bfn}(q)+1}$, and hence in~$\tB(\tau^{p_{\bfn}(q) + 1} \delta_0)$, by Schwarz' lemma we have
$$ \left| Df^{|\bfn|}(z) \right|
\ge
\frac{(\tau^{q+1} \delta_0 / 2)^{1/d}}{(\tau^{p_{\bfn}(q) + 1} \delta_0)^{1/d}}
\ge
\frac{1}{2^{1/d}} \tau^{- (p_{\bfn}(q) - q)/d}.$$
Since  $\#Z_{\bfn}(q)\le d^{k(\bfn)}$, the desired conclusion follows.
\end{proof}

In view of Lemma~\ref{l:contribution of a bad chain}, Proposition~\ref{prop:poinbad} is a direct consequence of the following lemma.

\begin{lem}
\label{l:estimate of bad}
Suppose~$f$ is backward contracting.
Then for every sufficiently large~$q \ge 0$ we have
$$ \sum_{\bfn \in \cS(q)} d^{k(\bfn)} \tau^{t \cdot p_{\bfn}(q) / q}
\le
2^{1+t/d} d (1 - 2^{-t/d})^{-1} \left( R(\tau^q \delta_0)^{-1} \tau^{q - 2} \right)^{t/d}.$$
\end{lem}
\begin{proof}
Put~$C_* := d (1 - 2^{- t/d})^{-1} 2^{t/d}$ and let~$q_0 \ge 0$ be sufficiently large so that for every~$q \ge q_0$ we have
\begin{equation}
  \label{e:a priori bc}
  R(\tau^q \delta_0) > \tau^{-2} \left( 2 C_* \right)^{d/t}.
\end{equation}
For each pair of integers~$q \ge 0$ and~$m \ge 1$, put
$$ \Xi_t(q, m)
:=
\sum_{\substack{\bfn \in \cS(q) \\ |\bfn| \le m}} d^{k(\bfn)} \tau^{t \cdot p_{\bfn}(q) / q}, $$
and~$\Xi_t(q, 0) := 0$.
To prove the lemma, it is enough to show that for every pair of integers~$q \ge q_0$ and~$m \ge 0$ we have
\begin{equation}
  \label{e:finite time estimate of bad}
\Xi_t(q, m)
\le
2 C_* \left( R(\tau^q \delta_0)^{-1} \tau^{q - 2} \right)^{t/d}.
\end{equation}
We proceed by induction in~$m$.
By definition, for every~$q$ we have~$\Xi_t(q, 0) = 0$, so inequality~\eqref{e:finite time estimate of bad} holds trivially when~$m = 0$.
Let~$m \ge 1$ be an integer and suppose that for every~$q \ge q_0$ the inequality~\eqref{e:finite time estimate of bad} holds with~$m$ replaced by~$m - 1$.
Let~$q \ge q_0$ be given and note that for every~$k \ge 2$, and every~$\bfn = (n_1, n_2, \ldots, n_k)$ in~$\cS(q)$, we have
$$ (n_2, \ldots, n_k) \in \cS(p_{n_1}(q))
\text{ and }
p_{(n_2, \ldots, n_k)}(p_{n_1}(q)) = p_{(n_1, \ldots, n_k)}(q). $$
Conversely, for every integer~$n$ in~$\cN(q)$, and every~$k' \ge 1$ and~$(n_1', \ldots, n_{k'}')$ in~$\cS(p_{n}(q))$, we have
$$ (n, n_1', \ldots, n_{k'}') \in \cS(q)
\text{ and }
p_{(n, n_1', \ldots, n_{k'}')}(q) = p_{(n_1', \ldots, n_{k'}')}(p_n(q)). $$
Thus,
\begin{equation*}
\begin{split}
\Xi_t(q, m)
& =
\sum_{\substack{n \in \cN(q) \\ n \le m}} d \left( \tau^{t \cdot p_n(q) / d} + \Xi_t \left( p_n(q), m - n \right) \right)
\\ & \le
\sum_{\substack{n \in \cN(q) \\ n \le m}} d \tau^{t \cdot p_n(q)/d} + d \sum_{\substack{n \in \cN(q) \\ n \le m  - 1}} \Xi_t \left( p_n(q), m - 1 \right).
\end{split}
\end{equation*}
Together with~\eqref{e:discretized child contribution}, and the induction hypothesis, this implies
\begin{multline*}
\Xi_t(q, m)
\le
C_* \left( R(\tau^q \delta_0)^{-1} \tau^{q - 2} \right)^{t/d}
+ d \sum_{\substack{n \in \cN(q) \\ n \le m - 1}} \Xi_t(p_n(q), m - 1)
\\ \le
C_* \left[ \left( R(\tau^q \delta_0)^{-1} \tau^{q - 2} \right)^{t/d}
+ 2d \sum_{\substack{n \in \cN(q) \\ n \le m - 1}} \left( R(\tau^{p_n(q)} \delta_0)^{-1} \tau^{p_n(q) - 2} \right)^{t/d} \right]
\end{multline*}
Using~\eqref{e:a priori bc}, and then~\eqref{e:discretized child contribution} again, we obtain
\begin{equation*}
  \begin{split}
\Xi_t(q, m)
& \le
C_* \left[ \left( R(\tau^q \delta_0)^{-1} \tau^{q - 2} \right)^{t/d}
+ (1 - 2^{- t/d}) 2^{-t/d} \sum_{\substack{n \in \cN(q) \\ n \le m}} \tau^{t \cdot p_n(q) / d} \right]
\\ & \le
2 C_* \left( R(\tau^q \delta_0)^{-1} \tau^{q - 2} \right)^{t/d}.
  \end{split}
\end{equation*}
This completes the proof of the induction step and of the lemma.
\end{proof}

\subsection{Proof of Theorem~\ref{thm:reduced}}
\label{ss:proof of reduced}
The proof of Theorem~\ref{thm:reduced} is at the end of this section, after a couple of lemmas.

Assume~$f$ is backward contracting, fix~$t > \HDhyp(J(f))$, and consider the notation introduced in~\S\ref{ss:thickened grandchildren} for this choice of~$t$.
Put~$\borbit \= \bigcup_{m=1}^\infty f^{-m}(0)$, and for each~$z$ in this set denote by~$m(z) \ge 1$ the unique integer~$m \ge 1$ such that~$f^m(z) = 0$.
Note that for every~$z$ in~$\borbit$ there is an integer~$q \ge 0$ such that the pull-back of~$V_q$ by~$f^{m(z)}$ containing~$z$ is conformal.
Denote by~$q(z)$ the least integer~$q \ge 0$ with this property.
\begin{lem}
\label{l:conformal before bad}
There is a constant~$C_{\&} > 0$ such that for every~$z$ in~$\borbit$ satisfying~$q(z) \ge 1$, and every~$q$ in~$\{0, \ldots, q(z) - 1 \}$, there exists~$\bfn$ in~$\cS(q)$ such that the following hold:
\begin{itemize}
\item
$m(z) \ge |\bfn|$, and~$\zeta(z) := f^{m(z) - |\bfn|}(z)$ is in~$Z_{\bfn}(q)$, and hence in~$V_{p_{\bfn}(q) + 1}$;
\item
$f^{m(z) - |\bfn|}$ maps a neighborhood~$U(z)$ of~$z$ conformally onto~$V_{p_{\bfn}(q)}$;
\item
Denoting by~$\zeta'(z)$ the unique point in~$U(z)$ such that $f^{m(z) - |\bfn|}(\zeta'(z))=0$, we have
$$|Df^{m(z)}(z)|
\ge
C_{\&} \left|Df^{|\bfn|}(\zeta(z)) \right| \left|Df^{m(z) - |\bfn|}(\zeta'(z)) \right|. $$
\end{itemize}
\end{lem}
\begin{proof}
The third assertion follows from the first and the second, together with Koebe distortion theorem.
To prove the first and second assertions, we proceed by induction in~$m(z)$.
Let~$z$ be a point in~$\borbit$ such that~$q(z) \ge 1$ and~$m(z) = 1$, and let~$q$ be in~$\{0, \ldots, q(z) - 1 \}$.
Then~$1$ is in~$\cN(q)$, and the desired assertions are easily seen to be satisfied with~$\bfn = (1)$.
Let~$m \ge 2$ be an integer and suppose the desired assertions are satisfied for every~$z$ in~$\borbit$ such that~$q(z) \ge 1$ and~$m(z) \le m - 1$, and for every~$q$ in~$\{0, \ldots, q(z) - 1 \}$.
Let~$z$ be a point in~$\borbit$ such that~$q(z) \ge 1$ and~$m(z) = m$, and let~$q$ be in~$\{0, \ldots, q(z) - 1 \}$.
Note that~$f(z)$ is in~$\borbit$ and $m(f(z)) = m(z) - 1$.
If~$q(f(z)) \le q$, then the pull-back of~$V_q$ by~$f^{m(z) - 1}$ containing~$f(z)$ is conformal.
This implies that~$m(z)$ is in~$\cN(q)$, and then the desired assertions are verified with~$\bfn = (m(z))$.
If~$q(f(z)) \ge q + 1$, then we can apply the induction hypothesis with~$z$ replaced by~$f(z)$; let~$\bfn' = (n_1', \ldots, n_{k'}')$ be the corresponding element of~$\cS(q)$.
If the pull-back of~$U(f(z))$ by~$f$ containing~$z$ is conformal, then the desired assertions are verified with~$\bfn = \bfn'$.
Otherwise, $n' \= m(z) - |\bfn'|$ is in~$\cN(p_{\bfn'}(q))$, and then the desired assertions are verified with~$\bfn = (n_1', \ldots, n_{k'}', n')$.
\end{proof}
\begin{lem}
\label{l:conformal Poincare}
Assume that $f$ is backward contracting.
Then for each $q \ge 0$ and $t > \HDhyp(J(f))$, we have
$$ \sum_{\substack{z \in \borbit \\ q(z) \le q}} |Df^{m(z)}(z)|^{-t}
<
\infty.$$
\end{lem}
\begin{proof}
Let~$\mu$ denote a conformal measure of~$f$ of exponent~$h \= \HDhyp(f)$.
For each~$z$ in~$\borbit$, let~$B(z)$ denote the pull-back of~$V_{q + 1}$ by~$f^{m(z)}$ that contains~$z$.
By Koebe distortion theorem, there is a constant~$C > 1$ such that for every~$z$ in~$\borbit$ satisfying~$q(z) \le q$, the distortion of~$f^{m(z)}$ on~$B(z)$ is bounded by~$C$.
Then for every such~$z$ we have $|Df^{m(z)}(z)| \ge C^{-1} \diam (B(z))^t$, and by conformality of~$\mu$, we also have
$$\mu(B(z)) \ge C^{-h} |Df^{m(z)}(z)|^{- h} \mu(V_{q + 1}).$$
Since for a fixed integer~$m \ge 1$ the sets~$(B(z))_{z \in f^{-m}(0)}$ are pairwise disjoint, we have
\begin{equation*}
  \begin{split}
\sum_{\substack{z\in f^{-m}(0)\\ q(z) \le q}} |Df^m(z)|^{-t}
& \le
C^t \mu(V_{q + 1})^{-1} \sum_{\substack{z\in f^{-m}(0)\\ q(z) \le q}} \diam (B(z))^{t-h} \mu(B(z))
\\ & \le
\mu(V_{q + 1})^{-1} \max_{z\in f^{-m}(0)} \diam (B(z))^{t-h}.
  \end{split}
\end{equation*}
By~\cite[Theorem~A]{RivShe1004}, for every~$\beta > 0$ the sequence~$\left( \max_{z\in f^{-m}(0)} \diam (B(z)) \right)_{m = 1}^{\infty}$ decreases faster than the sequence~$(m^{- \beta})_{m = 1}^{\infty}$.
The proposition follows summing over~$m \ge 1$.
\end{proof}

\begin{proof}[Proof of Theorem~\ref{thm:reduced}]
One of the implications of the theorem is given by Proposition~\ref{p:direct implication}.
To prove the reverse implication, fix~$t > \HDhyp(f)$ such that$$ \sum_{n = 1}^{\infty} |Df^n(c)|^{- t/d} < + \infty. $$
By part~$2$ of Theorem~\ref{thm:lb2bc} we have
$$ \sum_{q = 0}^{\infty} R(\tau^q \delta_0)^{- t/d} < + \infty. $$
Thus, if we denote by~$C_{\#}$ and~$C_{\&}$ the constants given by Proposition~\ref{prop:poinbad} and Lemma~\ref{l:conformal before bad}, respectively, then there is~$Q \ge 1$ so that
$$ C_{\#} C_{\&}^{-t} \sum_{q = Q + 1}^{\infty} R(\tau^{q-1}\delta_0)^{-t/d}
\le
\frac{1}{2}. $$
Taking~$Q$ larger if necessary, assume that for each~$z$ in~$f^{-1}(0)$ we have~$Q \ge q(z)$.
For every pair of integers~$q \ge 0$ and~$m \ge 1$, put
$$ \cP_t(q, m)
\=
\sum_{\substack{z \in \borbit \\ q(z) = q, m(z) \le m}} |Df^{m(z)}(z)|^{-t}, $$
and note that by Lemma~\ref{l:conformal Poincare} we have
$$ C' \= \sum_{q=0}^Q \sup \{ \cP_t(q, m) : m \ge 1 \}
<
\infty.$$
To prove that~$\cP(0, t)$ is finite, it is enough to prove that for every integer~$m$ we have
\begin{equation}
  \label{e:finite time poincare}
 \sum_{q = 0}^{+ \infty} \cP_t(q, m) \le 2 C'.
\end{equation}
We proceed by induction.
By our choice of~$Q$, for each~$z$ in~$f^{-1}(0)$ we have~$q(z) \le Q$.
So when~$m = 1$ inequality~\eqref{e:finite time poincare} follows from our definition of~$C'$.
Let~$m \ge 2$ be an integer and suppose~\eqref{e:finite time poincare} holds with~$m$ replaced by~$m - 1$.
Using the notation of Lemma~\ref{l:conformal before bad}, for each~$q \ge Q + 1$ we have
\begin{multline*}
\cP_t(q, m)
\\ \le
C_{\&}^{-t} \left( \sum_{\bfn \in \cS(q - 1)} \sum_{\zeta\in Z_{\bfn}(q-1)} |Df^{m(\zeta)}(\zeta)|^{-t} \right)
\cdot
\left( \sum_{\substack{\zeta' \in \borbit \\ m(\zeta') \le m - 1}} |Df^{m(\zeta')}(\zeta')|^{-t} \right).
\end{multline*}
So by Proposition~\ref{prop:poinbad} we have
$$\cP_t(q, m)
\le
C_{\#} C_{\&}^{-t} R(\tau^{q - 1} \delta_0)^{-t/d} \sum_{p=0}^\infty \cP_t(p, m - 1).$$
Summing over~$q \ge Q + 1$, we obtain by the induction hypothesis
$$ \sum_{q=Q + 1}^\infty \cP_t(q, m)
\le
\frac{1}{2} \sum_{p=0}^\infty \cP_t(p, m - 1)
\le
C'.$$
Using the definition of~$C'$, this implies~\eqref{e:finite time poincare}.
This completes the proof of the induction step and of the theorem.
\end{proof}

% For Bibtex bibliography:
% \bibliographystyle{alpha}
% \bibliography{$HOME/papers/0BIB/papers}

\end{document}